\documentclass[11pt,reqno]{amsart}
\usepackage{graphics}
\usepackage{amsmath}
\usepackage{latexsym}
\usepackage{amssymb}
\usepackage[all]{xy}
\xyoption{matrix}
\xyoption{arrow}
\usepackage{epsf}
\begin{document}

\newcommand{\lcm}{\operatorname{lcm}\nolimits}
\newcommand{\coker}{\operatorname{coker}\nolimits}
\renewcommand{\th}{\operatorname{th}\nolimits}
\newcommand{\rej}{\operatorname{rej}\nolimits}
\newcommand{\extto}{\xrightarrow}
\renewcommand{\mod}{\operatorname{mod}\nolimits}
\newcommand{\Sub}{\operatorname{Sub}\nolimits}
\newcommand{\ind}{\operatorname{ind}\nolimits}
\newcommand{\Fac}{\operatorname{Fac}\nolimits}
\newcommand{\add}{\operatorname{add}\nolimits}
\newcommand{\Hom}{\operatorname{Hom}\nolimits}
\newcommand{\Irr}{\operatorname{Irr}\nolimits}
\newcommand{\Rad}{\operatorname{Rad}\nolimits}
\newcommand{\RHom}{\operatorname{RHom}\nolimits}
\newcommand{\uHom}{\operatorname{\underline{Hom}}\nolimits}
\newcommand{\End}{\operatorname{End}\nolimits}
\renewcommand{\Im}{\operatorname{Im}\nolimits}
\newcommand{\Ker}{\operatorname{Ker}\nolimits}
\newcommand{\Coker}{\operatorname{Coker}\nolimits}
\newcommand{\Ext}{\operatorname{Ext}\nolimits}
\newcommand{\op}{{\operatorname{op}}}
\newcommand{\Ab}{\operatorname{Ab}\nolimits}
\newcommand{\id}{\operatorname{id}\nolimits}
\newcommand{\pd}{\operatorname{pd}\nolimits}
\newcommand{\ql}{\operatorname{q.l.}\nolimits}
\newcommand{\rank}{\operatorname{rank}\nolimits}
\newcommand{\A}{\operatorname{\mathcal A}\nolimits}
\newcommand{\C}{\operatorname{\mathcal C}\nolimits}
\newcommand{\D}{\operatorname{\mathcal D}\nolimits}
\newcommand{\E}{\operatorname{\mathcal E}\nolimits}
\newcommand{\X}{\operatorname{\mathcal X}\nolimits}
\newcommand{\Y}{\operatorname{\mathcal Y}\nolimits}
\newcommand{\F}{\operatorname{\mathcal F}\nolimits}
\newcommand{\Z}{\operatorname{\mathbb Z}\nolimits}
\renewcommand{\P}{\operatorname{\mathcal P}\nolimits}
\newcommand{\T}{\operatorname{\mathcal T}\nolimits}
\newcommand{\G}{\Gamma}
\renewcommand{\L}{\Lambda}
\newcommand{\bdot}{\scriptscriptstyle\bullet}

\newcommand{\mP}{\Phi^m_{\geq -1}}
\newcommand{\mD}{\Delta_m}
\newcommand{\mR}{R_m}
\newcommand{\mS}{\Sigma_m}

\renewcommand{\r}{\operatorname{\underline{r}}\nolimits}
\newcommand{\CBM}{\mathcal C_{BM}}
\newcommand{\bigamalg}{\amalg}
\newtheorem{lemma}{Lemma}[section]
\newtheorem{proposition}[lemma]{Proposition}
\newtheorem{cor}[lemma]{Corollary}
\newtheorem{thm}[lemma]{Theorem}
\newtheorem{rem}[lemma]{Remark}
\newtheorem{defin}[lemma]{Definition}
\newtheorem{example}[lemma]{Example}

\title{Coloured quiver mutation for higher cluster categories}

\author[Buan]{Aslak Bakke Buan}
\address{Institutt for matematiske fag\\
Norges teknisk-naturvitenskapelige universitet\\
N-7491 Trondheim\\
Norway}
\email{aslakb@math.ntnu.no}

\author[Thomas]{Hugh Thomas}\thanks{Both authors were supported by 
a STORFORSK-grant 167130 from the Norwegian Research Council;
the second author was also supported by an NSERC Discovery Grant.}
\address{Department of Mathematics and Statistics\\University of 
New Brunswick\\Fredericton NB\\E3B 1J4 Canada
}
\email{hthomas@unb.ca}
\maketitle

\begin{abstract}
We define mutation on coloured quivers associated to tilting objects in higher cluster
categories. We show that this operation is compatible with the mutation operation on the tilting objects.
This gives a combinatorial approach to tilting in higher cluster categories and especially an algorithm
to determine the Gabriel quivers of tilting objects in such categories.
\end{abstract}

\section*{Introduction}
A cluster category is a certain 2-Calabi-Yau orbit category of the derived category of a hereditary 
abelian category. Cluster categories were introduced in \cite{bmrrt} in order to give a categorical model for
the combinatorics of Fomin-Zelevinsky cluster algebras \cite{fz}. 
They are triangulated \cite{k}
and admit (cluster-)tilting objects,
which model the clusters of a corresponding (acyclic) cluster algebra \cite{ck}.
Each cluster in a fixed cluster algebra comes together with a finite quiver, and in the categorical model
this quiver is in fact the Gabriel quiver of the 
corresponding tilting object \cite{bmrt}.

A principal ingredient in the 
construction of a cluster algebra is 
quiver mutation. It controls the exchange procedure
which gives a rule for producing a new cluster variable and hence a new cluster from a given cluster. 
Exchange is modeled
by cluster categories in the acyclic case \cite{bmr2} in terms of a mutation rule for tilting objects, i.e. a rule
for replacing an indecomposable direct summand in a tilting object with another indecomposable
rigid object, to get a new tilting object. Quiver mutation describes the 
relation between the Gabriel quivers of the corresponding tilting objects.



Analogously to the definition of the cluster category,
for a positive integer $m$, it is natural to define a certain $m+1$-Calabi-Yau orbit 
category of the derived category of a hereditary abelian category. This is called
the {\em $m$-cluster category}. Implicitly, $m$-cluster categories was first studied in \cite{k}, 
and their (cluster-)tilting objects have been studied  
in \cite{abst,f,hj1,hj2,iy,kr1,kr2,t,w,z,zz}. Combinatorial descriptions of $m$-cluster categories 
in Dynkin type $A_n$ and $D_n$ are given 
in \cite{bm1,bm2}.   

In cluster categories the mutation rule for tilting objects
is described in terms of certain triangles called {\em exchange triangles}. 
By \cite{iy} the existence of exchange triangles generalizes to $m$-cluster categories. 
It was shown in \cite{zz,w} that there are exactly $m+1$ 
non-isomorphic complements to an almost complete
tilting object, and that
they are determined by the $m+1$ exchange triangles defined in \cite{iy}.

The aim of this paper is to give a combinatorial description of mutation
in $m$-cluster categories. {\it A priori}, one might expect to be able to
do this by keeping track of the Gabriel quivers of
the tilting objects.  However, it is easy to see that the Gabriel quivers
do not contain enough information.  

We proceed to associate to a tilting object a quiver each of whose arrows
has an associated colour $c\in \{0,\dots,m\}$.  The arrows with colour
0 form the Gabriel quiver of the tilting object. We then define a 
mutation operation on coloured quivers and show that it is compatible with 
mutation of tilting objects.  
A consequence is that the effect of an arbitrary sequence of mutations
on a tilting object in an $m$-cluster
category can be calculated by a purely combinatorial procedure.  

Our definition of a coloured quiver associated to a tilting object
makes sense in any $m+1$-Calabi-Yau category, such as for example those
studied in \cite{iy}. We hope that our constructions may shed some light
on mutation of tilting objects in this more general setting.

In section 1, we review some elementary facts about higher cluster
categories.
In section 2, we explain how to define the coloured quiver of a 
tilting object, we define coloured quiver mutation, and we state our
main theorem. In sections 3 and 4, we state some further lemmas about
higher cluster categories, and we prove certain properties of the
coloured quivers of tilting objects.  
We prove our main result in sections 5 and 6.  In sections 7 and 8 
we point out some applications. In section 9 we interpret our construction in terms
of $m$-cluster complexes. In section 10, we give an alternative
algorithm for computing coloured quiver mutation.  
Section 11 discusses the example of 
$m$-cluster categories of Dynkin type $A_n$, using the model developed
by Baur and Marsh \cite{bm1}.

We would like to thank Idun Reiten, in conversation with whom the initial idea
of this paper took shape.


\section{Higher cluster categories}\label{s:elem}

Let $K$ be an algebraically closed field, and let 
$\Gamma$ be a finite acyclic quiver with $n= n_{\Gamma}$ vertices.
Then the path algebra $H =K\Gamma$ is a hereditary finite dimensional basic $K$-algebra 

Let $\mod H$ be the category of finite dimensional left $H$-modules.
Let $\D = D^b(H)$ be the bounded derived category of $H$, and let $[i]$ be the $i$'th shift functor on $\D$. 
We let $\tau$ denote the Auslander-Reiten translate, which is an autoequivalence on $\D$ such that
we have a bifunctorial isomorphism in $\D$
\begin{equation}\label{ar}
\Hom(A,B[1]) \simeq D\Hom(B,\tau A).\end{equation}
In other words $\nu = [1] \tau$ is a Serre functor.

Let $G= \tau^{-1}[m]$.
The $m$-cluster category is the orbit category $\C =\C_m = \D/\tau^{-1}[m]$. 
The objects in $\C$ are the objects in $\D$, and two objects $X,Y$ are isomorphic in $\C$ if
and only if $X \simeq G^i Y$ in $\D$. 
The maps are given by $\Hom_{\C_m}(X,Y) = \amalg_{i \in \mathbb{Z}} \Hom_{\D}(X,G^iY)$.
By \cite{k}, the category $\C$ is triangulated and the canonical functor $\D \to \C$ is a triangle functor. We denote therefore
by $[1]$ the suspension in $\C$. The $m$-cluster category is also Krull-Schmidt and 
has an AR-translate $\tau$ inherited from $\D$, such that
the formula (\ref{ar}) still holds in $\C$. If follows that $\nu = [1] \tau$ is a Serre functor for $\C$ and that
$\C$ is $m+1$-Calabi-Yau, since $\nu \simeq [m+1]$.

The indecomposable objects in $\D$ are of the form $M[i]$, 
where $M$ is an indecomposable $H$-module and $i\in \mathbb{Z}$.  
We can choose a fundamental
domain for the action of $G= \tau^{-1}[m]$ on $\D$, 
consisting of the indecomposable
objects $M[i]$ with $0\leq i \leq m-1$, together with the objects
$M[m]$ with $M$ an indecomposable projective $H$-module.  Then each 
indecomposable object in $\C$ is isomorphic to exactly one of the 
indecomposables in this fundamental domain.  
We say that $M[d]$ has degree $d$, denoted $\delta(M[d]) = d$.
Furthermore, for an arbitrary object $X= \amalg X_i$ in  $\C_m$, we let $\Delta_d(X) = \amalg_j X_j[-d]$ be the $H$-module 
which is the (shifted) direct sum of all summands $X_j$ of $X$ with $\delta(X_j)= d$.

In the following theorem the equivalence between (i) and (ii) is shown in \cite{zz,w} and 
the equivalence between (i) and (iii) is shown in \cite{z}.
\begin{thm}
Let $T$ be an object in $\C$ satisfying $\Hom_{\C}(T,T[i]) = 0$ for $i=1,\dots,m$. Then the following are equivalent
\begin{itemize}
\item[(i)] If $\Hom_{\C}(T,U[i]) = 0$ for $i=1,\dots,m$ 
then $U$ is in $\add T$.
\item[(ii)] If $\Hom_{\C}(U \amalg T, U[i]) = 0$ for $i=1,\dots,m$ 
then $U$ is in $\add T$.
\item[(iii)] $T$ has $n$ indecomposable direct summands, up to isomorphism. 
\end{itemize}
\end{thm}
Here $\add T$ denotes the additive closure of $T$. A (cluster-)tilting object $T$ in an $m$-cluster 
is an object satisfying the conditions of the above Theorem.
For a tilting object $T =\amalg_{i=1}^v T_i$, with each $T_i$ indecomposable, and $T_k$ an indecomposable direct summand, 
we call $\bar{T} = T/T_k$ an almost complete tilting object. 
We let $\Irr_{\A}(X,Y)$ denote the $K$-space of irreducible maps $X \to Y$ in a Krull-Schmidt $K$-category $\A$.
The following crucial result is
proved in \cite{zz} and \cite{w}.

\begin{proposition} \label{p:number}
There are, up to isomorphism, $m+1$ complements of an almost complete tilting object. 
\end{proposition}

Let $T_k$ be an indecomposable direct summand in an $m$-cluster tilting object $T = \bar{T} \amalg T_k$.
The complements of $\bar{T}$ are denoted $T_k^{(c)}$ for $c= 0,1, \dots ,m$, where
$T_k = T_k^{(0)}$. 
By \cite{iy}, there are $m+1$ exchange triangles $$T_k^{(c)} \overset{f_k^{(c)}}{\to} B_k^{(c)} 
\overset{g_k^{(c+1)}}{\to} T_k^{(c+1)}  
\overset{h_k^{(c+1)}}{\to}.$$ Here the $B_k^{(c)}$ are in $\add (T/T_k)$ and the maps $f_k^{(c)}$ (resp. $g_k^{(c)}$) are 
minimal left (resp. right) $\add (T/T_k)$-approximations, and hence not split mono or split epi. Note that by minimality,
the maps $f_k^{(c)}$ and $g_k^{(c)}$ have no proper zero summands.

\section{Coloured quiver mutation}

We first recall the definition of quiver mutation, formulated in \cite{fz} in terms of 
skew-symmetric matrices. Let $Q = (q_{ik})$ be a quiver with vertices $1, \dots, n$ and with no loops or 
oriented two-cycles, where $q_{ik}$ denotes the number of arrows from $i$ to $k$. 
Let $j$ be a vertex in $Q$. 
Then, a new quiver $\mu_j(Q)= \widetilde{Q} = (\tilde{q}_{ik})$ is defined 
by the following data 
 
$$\tilde{q}_{ik} = 
\begin{cases}  q_{ki} & \text{  if $j =k$ or $j=i$} \\
		 \max \{0, q_{ik} - q_{ki} + q_{ij} q_{jk} - q_{kj} q_{ji} \} & \text{  if $i \neq j \neq k$} 
                 \end{cases}
$$
It is easily verified that this definition is equivalent to the one of Fomin-Zelevinsky.

Now we consider coloured quivers. Let $m$ be a positive integer. An $m$-coloured (multi-)quiver
$Q$ consists of vertices $1, \dots, n$ and coloured arrows $i \overset{(c)}{\longrightarrow} j$,
where $c \in \{0, 1, \dots, m \}$. Let $q_{ij}^{(c)}$ denote the number of arrows from 
$i$ to $j$ of colour $(c)$.

We will consider coloured quivers with the following additional conditions.

\begin{itemize}
\item[(I)] No loops: $q_{ii}^{(c)} = 0$ for all $c$.
\item[(II)] Monochromaticity: If $q_{ij}^{(c)} \neq 0$, then $q_{ij}^{(c')} = 0$ for $c \neq c'.$
\item[(III)] Skew-symmetry: $q_{ij}^{(c)} = q_{ji}^{(m-c)}$.
\end{itemize}

We will define an operation on a coloured quiver $Q$ satisfying  
the above conditions. Let $j$ be a vertex in $Q$ and let $\mu_j(Q) = \widetilde{Q}$ be the coloured quiver defined by

$$\tilde{q}_{ik}^{(c)} = 
\begin{cases}  q_{ik}^{(c+1)} & \text{  if $j =k$} \\
		 q_{ik}^{(c-1)} &\text{  if $j=i$} \\
		 \max \{0, q_{ik}^{(c)} - \sum_{t \neq c} q_{ik}^{(t)} + (q_{ij}^{(c)} - q_{ij}^{(c-1)}) q_{jk}^{(0)} 
		 + q_{ij}^{(m)} (q_{jk}^{(c)}  -q_{jk}^{(c+1)}) \} & \text{  if $i \neq j \neq k$} 
                 \end{cases}
$$

In an $m$-cluster category $\C$, for every tilting object $T = \amalg_{i=1}^n T_i$, with the $T_i$
indecomposable, we will define a corresponding $m$-coloured quiver $Q_T$,
as follows.

Let $T_i, T_j$ be two non-isomorphic indecomposable direct summands of the $m$-cluster tilting object $T$ and
let $r_{ij}^{(c)}$ denote the multiplicity of $T_j$ in $B_i^{(c)}$.
We define the $m$-coloured quiver $Q_T$ of $T$ to have vertices $i$ corresponding to indecomposable direct summands $T_i$,
and $q_{ij}^{(c)} = r_{ij}^{(c)}$. 
Note, in particular, that the $(0)$-coloured arrows are the arrows from
the Gabriel quiver for the endomorphism ring of $T$.

By definition, $Q_T$ satisfies condition
(I). We show in Section \ref{s:higher} that (II) is satisfied (this also follows from \cite{zz}), and in Section \ref{s:symmetry} 
that (III) is also satisfied.  

The aim of this paper is to prove the following theorem, which is a generalization of the main result
of \cite{bmr2}. 

\begin{thm}\label{t:main}
Let $T = \amalg_{i=1}^n T_i$ and $T' = T/T_j \amalg T_j^{(1)}$ be $m$-tilting objects, where
there is an exchange triangle $T_j \to B_j^{(0)} \to T_j^{(1)} \to$. Then $Q_{T'} = \mu_j (Q_{T})$.
\end{thm}

In the case $m=1$ the coloured quiver of a tilting object $T$ is given by
$q_{ij}^{(0)} = \bar{q}_{ij}$ and
$q_{ij}^{(1)} = \bar{q}_{ji}$ where
$\bar{q}_{ij}$ denotes the number of arrows in the Gabriel quiver of $T$.
Then coloured mutation of the coloured quiver corresponds to FZ-mutation of the Gabriel quiver.

\bigskip
{\bf Example: $A_3$, $m= 2$}

Let $\Gamma$ be $A_3$ with linear orientation, i.e. the quiver
$1 \leftarrow 2 \to 3$.

The AR-quiver of the 2-cluster category of $H = K\Gamma$ is
$$
\xymatrix@!@C=0.02cm@R=0.4cm{
 & P_1 \ar[dr] & & {I_3} \ar[dr] & & *++[o][F-]{P_3[1]} \ar[dr] & & I_1[1] \ar[dr] & & P_1[2]\ar[dr] \\
 P_2[2] \ar[ur] \ar[dr] & & {P_2} \ar[ur] \ar[dr] & & *+[o][F-]{I_2} \ar[ur] \ar[dr] & & P_2[1] \ar[ur] \ar[dr] & &  I_2[1]\ar[ur]\ar[dr] & &P_2[2]   \\
 & {P_3} \ar[ur] & & *+[o][F-]{I_1} \ar[ur] & & P_1[1] \ar[ur] & & I_3[1]\ar[ur]& &
P_3[2] \ar[ur]& & 
}
$$
The direct sum $T= I_1 \amalg I_2 \amalg P_3[1]$ of the encircled indecomposable objects
gives a tilting object.
Its coloured quiver is 
$$
\xymatrix{
I_1 \ar@<0.6ex>^{(0)}[r] & I_2 \ar@<0.6ex>^{(0)}[r] \ar@<0.6ex>^{(2)}[l] & 
P_3[1] \ar@<0.6ex>^{(2)}[l] 
}
$$ 
Now consider the exchange triangle $$I_2 \to P_3[1] \to I_3[1] \to$$
and the new tilting object $T' = P_1\amalg I_3[1] \amalg P_3[1]$.
The coloured quiver of $T'$ is 
$$\xymatrix{
I_1 \ar@<0.6ex>_{(0)}@/_2.5pc/[rr] \ar@<0.6ex>^{(1)}[r] & I_3[1] \ar@<0.6ex>^{(2)}[r] \ar@<0.6ex>^{(1)}[l] & 
P_3[1] \ar@<0.6ex>^{(0)}[l] 
\ar@<0.6ex>^{(2)}@/^3.5pc/[ll]
}$$

\section{Further background on higher cluster categories}\label{s:higher}

In this section we summarize some further known results about $m$-cluster categories.
Most of these are from \cite{z} and \cite{zz}. We include some proofs for the convenience of the reader.

Tilting objects in $\C = \C_m$ give rise to partial tilting modules in $\mod H$, where
a {\em partial tilting module} $M$ in $\mod H$, is a module with $\Ext^1_H(M,M) = 0$.

\begin{lemma}\label{l:partial}
\begin{itemize}
\item[(a)] When $T$ is a tilting object in $\C_m$, then each $\Delta_d(T)$ is a partial tilting module in $\mod H$.
\item[(b)] The endomorphism ring of a partial tilting module has no oriented cycles in its ordinary quiver.
\end{itemize}
\end{lemma}

\begin{proof}
(a) is obvious from the definition.
See \cite[Cor. 4.2]{hr} for (b).
\end{proof}

In the following note that degrees of objects are always considered with a fixed choice of fundamental domain,
and sums and differences of degrees are always computed modulo $m+1$.

\begin{lemma}[\cite{z,zz}]\label{l:div}
Assume $m > 1$.
\begin{itemize}
\item[(a)] $\End(X) \simeq K$ for any indecomposable exceptional object $X$.
\item[(b)] We have that $$\delta(T_i^{(c+1)})-\delta(T_i^{(c)}) 
\begin{cases}
= 1 & \text{if } \delta(T_i^{(c)}) =m \\
\leq 1 & \text{if } \delta(T_i^{(c)}) \not \in \{m-1,m \} \\ 
\leq 2 &\text{if } \delta(T_i^{(c)}) =m-1 
\end{cases}$$
\item[(c)] The distribution of degrees of complements is one of the following
\begin{itemize}
\item[-] there is exactly one complement of each degree, or
\item[-] there is no complement of degree $m$, two complements in one degree $d \neq m$, and exactly
one complement in all degrees $\neq d,m$.
\end{itemize}
\item[(d)] If $\Hom(T_i^{(c)}, T_i^{(c')}) \neq 0$, then $c' \in \{c, c+1,c+2 \}$.
\item[(e)] For $t \in \{1, \dots, m \}$ we have 
$$\Hom(T_i^{(c)}, T_i^{(c')}[t]) = \begin{cases} K & \text{ if $c'-c+t = 0 (\mod m+1)$} \\ 0 & \text{ else} \end{cases}$$
\end{itemize}
\end{lemma}

\begin{proof}
(a) follows from the fact that $\Hom_H(X,X) =K$ for exceptional objects and the definition of maps in a $m$-cluster 
category.

(b) follows from the fact that $\Hom(T_i^{(c+1)}, T_i^{(c)}[1]) \neq 0$, since in the exchange triangles, the 
$f_i^{(c)}$ are not split mono and (c) follows from (b).

Considering the two different possible distributions of complements, we obtain from (c) 
that if $m \geq 3$ and $c' \geq c+3$ and $c' \neq c-1$, then 
$\Hom(T_i^{(c)}, T_i^{(c')}) =  0$. Consider the case $c' = c-1$. We can assume $m>2$, since else the statement 
is void. Hence we can clearly assume that $\delta(T_i^{(c)}) = \delta(T_i^{(c-1)})$.
There is an exchange triangle induced from an exact sequence in $\mod H$,

$$T_i^{(c-1)} \to B_i^{(c-1)} \to T_i^{(c)} \to T_i^{(c-1)}[1].$$
It is clear that $\Hom(T_i^{(c-1)}[1], T_i^{(c-1)}) = 0$, since $m>2$. We claim that
also $\Hom(B_i^{(c-1)}, T_i^{(c-1)}) = 0$.
This holds since $B_i^{(c-1)} \amalg T_i^{(c-1)}$ is a partial tilting object in $H$, and so 
there are no cycles in the endomorphism ring, by Lemma \ref{l:partial}. Hence also $\Hom(T_i^{(c)}, T_i^{(c-1)}) =  0$
follows, and this finishes the proof for (d).

For (e) we first apply $\Hom(T_i^{(c+1)}, \ )$ to the exchange triangle
$$T_i^{(c)} \to B_i^{(c)} \to T_i^{(c+1)} \to $$
and consider the corresponding long-exact sequence, to obtain that
$$\Hom(T_i^{(c+1)}, T_i^{(c)}[t]) = \begin{cases} K & \text{ if $t = 1$} \\ 0 & \text{ if $t=0$ or $t \in \{2, \dots, m \}$} \end{cases}.$$
Now consider $\Hom(T_i^{(c+u)}, T_i^{(c)}[v])$.  
When $0<v \leq u \leq m$, we have that
\begin{multline*}
\Hom(T_i^{(c+u)}, T_i^{(c)}[v]) \simeq \Hom(T_i^{(c+u+1)}, T_i^{(c)}[v+1]) \simeq \\
\Hom(T_i^{(c-1)}, T_i^{(c)}[v+m-u]) \simeq \Hom(T_i^{(c)}, T_i^{(c-1)}[1+u-v]). \end{multline*}

When $m \geq v>u \geq 0$, we have that
$$
\Hom(T_i^{(c+u)}, T_i^{(c)}[v]) \simeq \Hom(T_i^{(c+u-1)}, T_i^{(c)}[v-1]) \simeq \\
\Hom(T_i^{(c)}, T_i^{(c)}[v-u]). $$

Combining these facts, (e) follows.
\end{proof}

\begin{lemma}\label{l:div2}
The following statements are equivalent
\begin{itemize}
\item[(a)] $\Hom(T_i^{(1)}, T_j^{(1)}[1]) = 0$ 
\item[(b)] $T_j$ is not a direct summand in $B_i^{(m)}$
\item[(c)] $T_i$ is not a direct summand in $B_j^{(0)}$
\end{itemize}
Furthermore, $\Hom(T_i^{(c)}, T_j^{(1)}[1]) = 0$ for $c \neq 1$.
\end{lemma}
\begin{proof}
Note that $r_{ji}^{(0)} = r_{ij}^{(m)} = \dim \Irr_{\add T}(T_j,T_i)$, so (b)
and (c) are equivalent.
Consider the exact sequence
\begin{multline*} \Hom(T_i^{(c)}, T_j^{(0)}[1]) \to \Hom(T_i^{(c)}, B_j^{(0)}[1]) \to \\
 \Hom(T_i^{(c)}, T_j^{(1)}[1])
\to \Hom(T_i^{(c)}, T_j^{(0)}[2]) \to 
\end{multline*}
coming from applying $\Hom(T_i^{(c)}, \ )$ to the exchange triangle $$T_j^{(0)} \to B_j^{(0)} \to T_j^{(1)}.$$
The first and fourth terms are always zero. Using \ref{l:div}(e) we get that the second term (and hence the third) is non-zero 
if and only if 
$c=1$ and $T_i$ is a direct summand in $B_j^{(0)}$. 
\end{proof}

\begin{lemma}(\cite{iy,zz})\label{l:composing}
For $0 \leq l \leq m$, 
the composition $$\gamma_k^{(v,l)} = h_k^{(v)} \circ h_k^{(v-1)}[1] 
\circ h_k^{(v-2)}[2] \circ \cdots \circ  h_k^{(v-l+1)}[l-1]
\colon T_k^{(v)} \to T_k^{(v-l)}[l]$$ is 
non-zero  
and a basis for  
$\Hom(T_k^{(v)}, T_k^{(v-l)}[l])$.
\end{lemma}

\begin{proof}
For $m=1$, see \cite{bmrrt}.
Assume $m \geq 2$. For the first claim see \cite{iy}, while 
the second claim then follows from Lemma \ref{l:div}(e).
\end{proof}

We include an independent proof of the following crucial property.

\begin{proposition}\cite{zz} \label{l:disjoint}
$B_k^{(u)}$ and $B_k^{(v)}$ has no common non-zero direct summands whenever $u \neq v$.
\end{proposition}

\begin{proof}
When $m=1$, this is proved in \cite{bmr2}. Assume $m>1$.
We consider two cases, $\left|u-v \right| = 1$ or  $\left|u-v \right| > 1$.

Consider first the case $\left| u-v \right|=1$. Without loss of generality we can assume $u= 0$ and $v=1$,
and that $\delta(T_k^{(0)}) = 0$. 
Assume that there exists a (non-zero) indecomposable $T_x$, which is a direct summand in $B_k^{(0)}$ and in $B_k^{(1)}$. 
We have that $\delta(T_k^{(1)}) \in \{0,1 \}$ by Lemma \ref{l:div}(b). 
Assume first $\delta(T_k^{(1)}) =0 $. Then the exchange triangle 
$$T_k^{(0)} \to B_k^{(0)} \to T_k^{(1)} \to$$
is induced from the degree 0 part of the derived category, and hence from an exact sequence in $\mod H$. 
Then the endomorphism ring of the partial tilting module
$T_x \amalg T_k^{(1)}$ has a cycle, which is a contradiction to Lemma \ref{l:partial}.
Assume now that $\delta(T_k^{(1)}) =1$. Then $\delta({T_k^{(2)}}) \in \{0,1,2 \}$, where 0 can only occur if $m=2$.
If $\delta({T_k^{(2)}}) \in \{1,2 \}$, then 
clearly $\delta(T_x)=1$, and hence the partial tilting module $T_k^{(1)} \amalg T_x$ contains a cycle,
which is a contradiction. Assume that $\delta(T_k^{(2)}) = 0$ (and hence $m=2$). Then $\delta(T_x) \in \{0, 1 \}$.
If $\delta(T_x)=1$, we get a contradiction as in the previous case. If  $\delta(T_x)=0$, consider
the exchange triangle
$$T_k^{(2)} \to B_k^{(2)} \to T_k^{(0)} \to$$
which is induced from an exact sequence in $\mod H$. Hence there is a {\em non-zero} map $T_x \to B_k^{(2)}$ obtained by
composing $T_x \to T_k^{(2)}$ with the monomorphism $T_k^{(2)} \to B_k^{(2)}$, and thus there 
are cycles in the endomorphism ring of the partial tilting module $T_x \amalg B_k^{(2)} \amalg T_k^{(0)}$, a contradiction.
This finishes the case with $\left|u-v \right|=1$.

Assume now that $\left|u-v \right|>1$. Then we have $m > 2$.
Since  $\Hom(T_k^{(v)}, T_x) \neq 0$ and $\Hom(T_k^{(u)}, T_x) \neq 0$,
we have by Lemma \ref{l:div}(c) that $\left| v-u \right| \leq 2$. So without loss of
generality we can assume $v = u -2$. Assume that $\delta(T_k^{(u)}) = 0$.
Then $\delta(T_k^{(v)}) = m-1$ using Lemma \ref{l:div}(c) and the fact that $\Hom(T_k^{(v)}, T_x) \neq 0$.
Then also $\delta(T_k^{(v)}) \leq  m$. But $\Hom(T_x, T_k^{(v+1)}) \neq 0$, so
$\delta(T_x) \leq m$, contradicting the fact that $\Hom(T_k^{(u)}, T_x) \neq 0$.
\end{proof}

\begin{cor} $Q_T$ satisfies condition (II).  \end{cor}

\section{Symmetry}\label{s:symmetry}

Let $T=\widetilde T \amalg T_i \amalg T_j$ be a tilting object.  
In this section we show that the coloured quiver $Q_T$ 
satisfies condition 
(III). 

\begin{proposition}\label{p:symmetry}
With the notation of the previous section, we have $r_{ji}^{(c)} = r_{ij}^{(m-c)}$.
\end{proposition}

\begin{proof}

By Lemma \ref{l:div2}
we only need to consider the case $c \not \in \{0,m \}$. 
It is enough to show that $r_{ji}^{(c)} \leq r_{ij}^{(m-c)}$.

We first prove
\begin{lemma}\label{l:non-van}
Let $\alpha \colon T_j^{(c)} \to T_i$ be irreducible in 
$\add ((T/T_j) \amalg T_j^{(c)})$.
Then the composition $\alpha[-c] \circ \gamma_i^{(0,c)}[-c]  \colon T_j^{(c)}[-c] \to T_i^{(m-c+1)}$
is non-zero.
\end{lemma}

\begin{proof}
We have already assumed $c \neq 0$.
Assume $$\alpha[-c] \circ h_i^{(0)}[-c] \colon T_j^{(c)}[-c] \to T_i^{(0)}[-c] \to T_i^{(m)}[-c+1]$$ is zero. 
This means that $T_j^{(c)} \to T_i$ must factor through $B_i^{(m)} \overset{g_i^{(0)}}{\to} T_i$. 
Since $T_i$ is by assumption a summand in $B_j^{(c)}$, we have
that $T_i$ is not a summand in $B_j^{(0)}$ by Proposition \ref{l:disjoint}. 
Since $r_{ij}^{(m)} = r_{ji}^{(0)} = 0$, 
we have that $T_j$ is not a direct summand in $B_i^{(m)}$. 
This means that $\alpha$ is not irreducible in $\add ((T/T_j) \amalg T_j^{(c)})$, a contradiction. 
So $\alpha[-c] \circ h_i^{(0)}[-c] \colon T_j^{(c)}[-c] \to T_i^{(m)}[-c+1]$ is non-zero.

Assume $c>1$.
If the composition $\alpha[-c] \circ h_i^{(0)}[-c] \circ h_i^{(m)}[-c+1]$ is zero, then 
$\alpha[-c] \circ h_i^{(0)}[-c]$ factors through $$B_i^{(m-1)}[-c+1] \to T_i^{(m)}[-c+1].$$
We claim that $\Hom(T_j^{(c)}[-c], B_i^{(m-1)}[-c+1]) \simeq \Hom(T_j^{(c)}, B_i^{(m-1)}[1]) = 0$.
This clearly holds if $T_j$ is not a summand of $B_i^{(m-1)}$. In addition we have that
$\Hom(T_j^{(c)}, T_j[1]) = 0$ since $c>1$, using Lemma \ref{l:div}(e).
This is a contradiction, and
this argument can clearly be iterated to see that 
$\alpha[-c] \circ \gamma_i^{(0,c)}[-c]  \colon T_j^{(c)}[-c] \to T_i^{(m-c+1)}$
is non-zero, using Lemma \ref{l:div}(e).
\end{proof}

We now show that any irreducible map $\alpha \colon T_j^{(c)} \to T_i$ gives rise
to an irreducible map $\delta \colon T_i^{(m-c)} \to T_j$.

Consider the composition $$B_j^{(c-1)}[-c] \overset{g_j^{(c)}[-c]}{\longrightarrow} T_j^{(c)}[-c] \longrightarrow T_i^{(m-c+1)}.$$ 
Since $T_i$ is a summand in $B_j^{(c)}$ by assumption, it is not a 
summand in $B_j^{(c-1)}$.  Thus, $B_j^{(c-1)}$ is in 
$\add \widetilde T$. Since $\Hom(X, T_i^{(m-c+1)} [c])= 0$ for any $X$ in $\add \tilde{T}$, the composition vanishes.  

Using the exchange triangle 
$$B_j^{(c-1)}[-c] \overset{g_j^{(c)}[-c]}{\longrightarrow} T_j^{(c)}[-c] \overset{h_j^{(c)}[-c]}{\longrightarrow} T_j^{(c-1)}[-c+1],$$
we see that
$\alpha[-c] \circ \gamma_i^{(0,c)}[-c] \colon T_j^{(c)}[-c] \to T_i^{(m-c+1)}$
factors through the map $ T_j^{(c)}[-c] \overset{h_j^{(c)}[-c]}{\longrightarrow} T_j^{(c-1)}[-c+1]$, 
i.e. there is a commutative diagram

$$
\def \labelstyle{\scriptstyle}
\xymatrix@C=1.8cm{
B_j^{(c-1)}[-c] \ar^{g_j^{(c)}[-c]}[r]  & T_j^{(c)}[-c] \ar^{h_j^{(c)}[-c]}[r] \ar[d] &  T_j^{(c-1)}[-c+1] \ar[r] 
\ar^{\phi_1}[dl]&  \\
& T_i^{(m-c+1)} & &
}
$$

Similarly, using the exchange triangle  
$$B_j^{(c-2)}[-c+1] \overset{g_j^{(c-1)}[-c+1]}{\longrightarrow} T_j^{(c-1)}[-c+1] \overset{h_j^{(c-1)}[-c+1]}{\longrightarrow} T_j^{(c-2)}[-c+2]$$
we obtain a map $\phi_2 \colon T_j^{(c-2)}[-c+2] \to  T_i^{(m-c+1)}$ 

Repeating this argument $c$ times we obtain a map $\phi_c \colon T_j \to T_i^{(m-c+1)}$, such that
$\gamma_j^{(c,c)}[-c] \circ \phi_c = \alpha[-c] \circ \gamma_i^{(0,c)}$.

$$
\def \labelstyle{\scriptscriptstyle}
\def \objectstyle{\scriptscriptstyle}
\xymatrix@R=1.1cm@C=2.1cm{
T_j^{(c)}[-c] \ar_{h_j^{(c)}[-c]}[d] \ar[r] & T_i^{(m-c+1)}  \\  
T_j^{(c-1)}[-c+1] \ar_{h_j^{(c-1)}[-c+1]}[d] \ar^{\phi_1}[ur] & \\
T_j^{(c-2)}[-c+2] \ar_{h_j^{(c-2)}[-c+2]}[d] \ar^{\phi_2}[uur] & \\
\vdots \ar[d] &  \\ 
T_j   \ar^{\phi_c}[uuuur] & \\
&
}
$$

We claim that

\begin{lemma}\label{l:irred}
There is a map $\beta \colon T_j \to T_i^{(m-c+1)}$, such that 
$\gamma_j^{(c,c)}[-c] \circ \beta = \alpha[-c] \circ \gamma_i^{(0,c)}$,
and such that $\beta$ is irreducible in $\add ((T/T_i) \amalg T_i^{(m-c+1)})$.
\end{lemma}

\begin{proof}
Let $$T_j \overset{\begin{pmatrix} \psi' & \psi''  \end{pmatrix}}{\longrightarrow} (T_i^{(m-c+1)})' \amalg \widetilde{T}'$$ be a 
minimal left
$\add (T_i^{(m-c+1)} \amalg \widetilde{T})$-approximation, 
with $\widetilde{T}'$ in $\add \widetilde{T}$ and $(T_i^{(m-c+1)})'$ in $\add T_i^{(m-c+1)}$.
Let $\phi_c$ be as above, and factor it as 
$$T_j  \overset{\begin{pmatrix} \psi' & \psi'' \end{pmatrix}}{\longrightarrow} 
(T_i^{(m-c+1)})' \amalg \widetilde{T}'
\overset{\begin{pmatrix} \epsilon' \\ \epsilon'' \end{pmatrix}}{\longrightarrow} T_i^{(m-c+1)}.$$ 
Since $\gamma_j^{(c,c)}$ factors through $T_j^{(1)}[-1]$, we have that 
$\gamma_j^{(c,c)}[-c] \psi'' = 0$, 
so we have 
$$\gamma_j^{(c,c)}[-c](\psi' \epsilon' + \psi'' \epsilon'')= \gamma_j^{(c,c)}[-c] \psi' \epsilon'.$$
Hence, let we let $\beta = \psi' \epsilon'$ and since the summands in $\epsilon'$ are isomorphisms,
it is clear that $\beta$ is irreducible.
\end{proof}

Next, assume $\{ \alpha_t \}$ is a basis for the space of irreducible maps from $T_j^{(c)}$ to $T_i$.
Then, by Lemma \ref{l:non-van} the set $\{ \alpha_t \circ \gamma_i^{(0,c)} \}$ is also linearly independent. 
For each $\alpha_t$, consider the corresponding map $\beta_t$, such that 
$\gamma_j^{(c,c)}[-c] \circ \beta_t = \alpha_t[-c] \circ \gamma_i^{(0,c)}$, and
which we by Lemma \ref{l:irred} 
can assume is irreducible.
Assume a non-trivial linear combination $\sum k_t \beta_t$ is zero. Then also $\sum k_t (\gamma_j^{(c,c)}[-c] \circ \beta_t) = 
\sum k_t \alpha_t \circ \gamma_i^{(0,c)}=0$. But this contradicts Lemma \ref{l:non-van} since $\sum k_t \alpha_t$ is irreducible.
Hence it follows that 
$\{ \beta_t \}$ is also linearly independent. 
Hence, in the exchange triangle $T_i^{(m-c)} \to B_i^{(m-c)} \to T_i^{(m-c+1)}$, we have that $T_j$ appears
with multiplicity at least $r_{ji}^{(c)}$ in $B_i^{(m-c)}$.
So, we have that $r_{ji}^{(c)} \leq r_{ij}^{(m-c)}$, and the proof of the proposition
is complete.
\end{proof}

\section{Complements after mutation}

In this section we show how mutation in the vertex $j$ affects
the complements of the almost complete tilting object $T/T_i$.
As before, let $T = \widetilde{T} \amalg T_i \amalg T_j$ be an 
$m$-tilting object, and let $T' = T/T_j \amalg T_j^{(1)}$.

We need to consider
$$
\xymatrix{
T_i \ar_{(c)}@/_1pc/[rr] \ar^{(e)}[r] & T_j \ar^{(d)}[r] & T_k
}
$$ 
for all possible values of $c, d, e$.
However, we
have the following restriction on the colour of 
arrows.

\begin{proposition}\label{p:limits}
Assume $q_{ij}^{(e)}>0, q_{jk}^{(0)}>0$ and $q_{ik}^{(c)}>0$. Then $c \in \{e,e+1\}$. 
\end{proposition}

\begin{proof}
Consider the exchange triangle $T_i^{(e)} \to T_j \amalg X' \to T_i^{(e+1)} \to$. Note
that $T_j$ is a direct summand in the middle term $B_i^{(e)}$ by the assumption that $q_{ij}^{(e)}>0$. 
Consider also the exchange triangle $T_i^{(c)} \to T_k \amalg Z \to T_i^{(c+1)} \to$.
Pick an arbitrary non-zero map $h \colon T_j \to T_k$, and consider the map 
$\left( \begin{smallmatrix} h & 0 \\ 0 & 0 \end{smallmatrix} \right) \colon T_j \amalg X' \to T_k \amalg Z$.  
It suffices to show that whenever $c \not \in \{e,e+1 \}$, then $h$ is not irreducible in $\add T$. 
So assume that $c \not \in \{e,e+1 \}$. We claim that there is a commutative diagram 
$$
\xymatrix{
T_i^{(e)} \ar[r] \ar[d] &  T_j \amalg X' \ar[r] \ar^{\left( \begin{smallmatrix} h & 0 \\ 0 & 0 \end{smallmatrix} \right)}[d] &  
T_i^{(e+1)} \ar[r] \ar[d] & \\ 
T_i^{(c)} \ar[r] & T_k \amalg Z \ar[r] & T_i^{(c+1)} \ar[r] & 
}
$$
where the rows are the exchange triangles. 
The composition $T_i^{(e)} \to T_j \xrightarrow{h} T_k \to T_i^{(c+1)}$ is zero since 
\begin{itemize}
\item[-] if $c \neq e-1$ $\Hom(T_i^{(e)},T_i^{(c+1)}) = 0$ by using $c \not \in \{e,e+1 \}$ and Lemma \ref{l:div}(e)  
\item[-] if $c= e-1$, there is no non-zero composition $T_i^{(e)} \to T_j \to T_k \to T_i^{(c+1)} = T_i^{(e)}$ 
\end{itemize}
Hence the leftmost vertical map exists, and
then the rightmost map exists, using that $\C$ is a triangulated category.
Then, since $\Hom(T_i^{(e+1)}[-1], T_i^{(c)}) = 0$ by Lemma \ref{l:div}(e), 
there is a map $T_j \amalg X' \to T_i^{(c)}$, such that
$T_i^{(e)} \to T_i^{(c)} = T_i^{(e)} \to T_j \amalg X' \to T_i^{(c)}$. 
Hence there is map $T_i^{(e+1)} \to T_k \amalg Z$ such that
$T_j \amalg X' \to T_k \amalg Z = (T_j \amalg X' \to T_i^{(c)} \to T_k \amalg Z) 
+ (T_j \amalg X' \to T_i^{(e+1)} \to T_k \amalg Z)$.
By restriction we get 
\begin{equation}\label{factor}
h \colon T_j \to T_k = (T_j \to T_i^{(c)} \to T_k) + (T_j \to T_i^{(e+1)} \to T_k).
\end{equation}

Under the assumption $c \not \in \{e, e+1 \}$ we have that
$T_i^{(e+1)} \to T_k$ cannot be irreducible in $\add ((T/T_i) \amalg T_i^{(e+1)})$. Hence 
$T_i^{(e+1)} \to T_k = T_i^{(e+1)} \to B_i^{(e+1)} \to T_k$, where $T_k$ is not summand in $B_i^{(e+1)}$. 
Also, by Proposition
\ref{l:disjoint} we have that $T_j$ is not a summand in $B_i^{(e+1)}$. 
If $T_j \to T_i^{(c)}$ was irreducible in $\add ((T/T_i) \amalg T_i^{(c)})$, 
then there would be an irreducible map $T_i^{(c-1)} \to T_j$
in $\add ((T/T_i) \amalg T_i^{(c-1)})$,
and since $c \neq e+1$, this does not hold, by Proposition \ref{l:disjoint}.
Hence, $T_j \to T_i^{(c)} = T_j \to B_i^{(c-1)} \to T_i^{(c)}$, where $T_j$ is not a direct summand of
$B_i^{(c-1)}$. 
Also by Proposition \ref{l:disjoint} we have that $T_k$ is not a summand of $B_i^{(c-1)}$.
By (\ref{factor}), this shows that $h \colon T_j \to T_k$ is not irreducible in $\add T$.
\end{proof}

Let $T' = (T/T_j) \amalg T_j^{(1)}$. For $i \neq j$, let $(T_i^{(u)})'$ denote the complements
of $T'/T_i$, where there are exchange triangles
$$(T_i^{(u)})' \to (B_i^{(u)})' \to (T_i^{(u+1)})' \to$$

We first want to compare $(T_i^{(u)})'$ with $(T_i^{(u)})$.

\begin{lemma}\label{l:samecomp}
Assume that $q_{ij}^{(u)} = 0$ for $0 \leq u <c$ and that  $q_{ij}^{(m)} = 0$.
\begin{itemize}
\item[(a)] For $u= 0,1, \dots, c-1$,
the minimal left $\add (T/T_i)$-approximation $T_i^{(u)} \to B_i^{(u)}$
is also an $\add (T'/T_i)$-approximation.
\item[(b)] For $u= 0,1, \dots, c$, we have $(T_i^{(u)})' = T_i^{(u)}$. 
\end{itemize}
\end{lemma}

\begin{proof}
By assumption $T_j$ is not a direct summand in any of the $B_i^{(u)}$.
Assume there is a map $T_i^{(u)} \to T_j^{(1)}$ and consider the diagram 

$$
\xymatrix{
T_i^{(u+1)}[-1] \ar[r] &  T_i^{(u)} \ar[r] \ar [d] & B_i^{(u)} \ar [r] & \\
 & T_j^{(1)} & &
}
$$
Since $\Hom(T_i^{(u+1)},T_j^{(1)}[1])= 0$ by Lemma \ref{l:div2}, we see that 
the map $T_i^{(u)} \to T_j^{(1)}$ factors through $T_i^{(u)} \to B_i^{(u)}$.
Hence the minimal left $\add (T/T_i)$-approximation $T_i^{(u)} \to B_i^{(u)}$
is also an $\add (T'/T_i)$-approximation, so we have proved (a). Then
(b) follows directly.
\end{proof}

\begin{lemma}\label{l:comp}
Assume that $e \neq m$ and there are exchange triangles
\begin{equation}\label{i-tri}
T_i^{(e)} \to (T_j)^p \amalg X \to T_i^{(e+1)} \to
\end{equation}
and
\begin{equation}\label{j-tri}
T_j \to (T_k)^q \amalg Y \to T_j^{(1)} \to,
\end{equation}
where $p = q_{ij}^{(e)} > 0$  and $q = q_{jk}^{(0)} \geq 0$, i.e. $B_i^{(e)} = (T_j)^p \amalg X$ and 
$B_j^{(0)} = (T_k)^q \amalg Y$, where $T_k$ is not isomorphic to any direct summand in $Y$.

\begin{itemize}
\item[(a)]
The composition $T_i^{(e)} \to (T_j)^p \amalg X \to  (T_k)^{pq} \amalg Y^p \amalg X$ is a 
left $\add (T'/T_i)$-approximation.  
\item[(b)]
There is a triangle 
$$T_i^{(e)} \to (T_k)^{pq} \amalg Y^p \amalg X \to (T_i^{(e+1)})' \amalg C' \to$$ 
with $C'$ in $\add (T/(T_i \amalg T_j))$ and $T_i^{(e)}= (T_i^{(e)})'$.
\item[(c)]
There is a triangle $T_i^{(e+1)} \to (T_i^{(e+1)})' \amalg C' \to (T_j^{(1)})^p \to$.
\end{itemize}
\end{lemma}

\begin{proof}
Consider an arbitrary map $f \colon T_i^{(e)} \to U$ with $U$ in $\add (T'/T_i)$. 
We have that $\Hom(T_i^{(e+1)}, T_j^{(1)}[1]) = 0$, by Lemma \ref{l:div2}. 
Hence, by applying $\Hom(\ , U)$ to the triangle
(\ref{i-tri}) we get that $f$ factors through $T_i^{(e)} \to (T_j)^p \amalg X$. By applying $\Hom(\ , U)$ to the triangle
(\ref{j-tri}), and using that  $\Hom(T_j^{(1)}, T_j^{(1)}[1]) = 0$, we get that $f$ factors through 
$T_i^{(e)} \to (T_j)^p \amalg X \to Y^p \amalg (T_k)^{pq} \amalg X$. This proves (a).
For (b) and (c) we use the exchange triangles (\ref{i-tri}) and (\ref{j-tri}) and the octahedral axiom 
to obtain the commutative diagram
of triangles

$$
\xymatrix{
  T_i^{(e)} \ar[r] \ar@{=}[d] & (T_j)^p \amalg X \ar[r] \ar[d]            & T_i^{(e+1)} \ar[d] \ar[r] & \\ 
  T_i^{(e)} \ar[r]        & (T_k)^{pq} \amalg Y^p \amalg X \ar[r]  \ar[d] & C \ar[r] \ar[d]         &  \\
                    & (T_j^{(1)})^p \ar@{=}[r]                    & (T_j^{(1)})^p \ar[r]    & 
}
$$
By (a) the map $T_i^{(e)} \to (T_k)^{pq} \amalg Y^p \amalg X$ is a left $\add (T'/T_i)$-approximation,
and by Lemma \ref{l:samecomp} we have that $(T_i^{(e)})' = T_i^{(e)}$.
Hence $C = (T_i^{(e+1)})' \amalg C'$, where $C'$ is in $\add ((T_k)^{pq} \amalg Y^p \amalg X) \subset \add (T/(T_i \amalg T_j))$,
and with no copies isomorphic to $T_k$ in $Y$.
\end{proof}

Note that the induced $\add (T'/T_i)$-approximation
is in general not minimal.

\begin{lemma}\label{l:modtri}
Assume $e \neq m$ and $q_{ij}^{(e)} > 0$.
\begin{itemize}
\item[(a)]  Then there is a triangle 
$$(T_i^{(e+1)})' \amalg C' \xrightarrow{\alpha} B_i^{(e+1)} \amalg (T_j^{(1)})^p \to T_i^{(e+2)} \to $$     
where $\alpha$
is a minimal left $\add (T'/T_i)$-approximation, and $C'$ is as in Lemma \ref{l:comp}.
\item[(b)] There is an induced exchange triangle  
$$(T_i^{(e+1)})' \to \frac{B_i^{(e+1)} \amalg (T_j^{(1)})^p}{\alpha(C')} \to T_i^{(e+2)} \to $$     
where $\alpha(C') \simeq C'$.
\item[(c)] $(T_i^{(e+2)})' \simeq T_i^{(e+2)}$.
\end{itemize}
\end{lemma}

\begin{proof}
Consider the exchange triangle $$T_i^{(e+2)}[-1] \to T_i^{(e+1)} \to B_i^{(e+1)} \to $$
and the triangle from Lemma \ref{l:comp} (b)
\begin{equation}\label{comp-app}
T_i^{(e+1)} \to (T_i^{(e+1)})' \amalg C' \to (T_j^{(1)})^p \to.
\end{equation}
Apply the octahedral axiom, to obtain the commutative diagram of triangles

$$
\xymatrix{
  T_i^{(e+2)}[-1] \ar[r] \ar@{=}[d] & T_i^{(e+1)} \ar[r] \ar[d]            & B_i^{(e+1)} \ar[d] \ar[r] & \\ 
  T_i^{(e+2)}[-1] \ar[r]            & (T_i^{(e+1)})' \amalg C' \ar[r] \ar[d]  & G \ar[r] \ar[d]         & \\
                                  & (T_j^{(1)})^p \ar@{=}[r]           & (T_j^{(1)})^p \ar[r]    & 
}
$$
Since $T_j$ does not occur as a summand in
$B_i^{(e+1)}$ by Proposition \ref{l:disjoint}, 
we have that $\Hom(T_j^{(1)}, B_i^{(e+1)}[1]) = 0$. 
Hence the rightmost triangle splits, so we have a triangle
\begin{equation}\label{octa-tri}
T_i^{(e+2)}[-1] \to (T_i^{(e+1)})' \amalg C' \to B_i^{(e+1)} \amalg (T_j^{(1)})^p \to     
\end{equation}
By Lemma \ref{l:div2} we have that $\Hom(T_i^{(e+2)}, T_j^{(1)}[1])= 0$. 
By Lemma \ref{l:div}(e) we get that $\Hom(T_i^{(e+2)}, T_i[1]) = 0$, and clearly
$\Hom(T_i^{(e+2)}, T_l [1]) = 0$, for $l \neq i$. 
We hence get that all maps $(T_i^{(e+1)})' \amalg C' \to U$, with $U$ in $\add T'$, factor through
$(T_i^{(e+1)})' \amalg C' \to B_i^{(e+1)} \amalg (T_j^{(1)})^p$. Minimality is clear from the triangle
(\ref{octa-tri}).
This proves (a), and (b) follows from the fact that $C'$ contains no copies of $T_j$, and hence splits off.
(c) is a direct consequence of (b).
\end{proof}

\begin{proposition}\label{p:summarize}

\begin{itemize}
\item[(a)] If $q_{ij}^{(u)} = 0$ for $u = 0, \dots, m$, then 
$(T_i^{(v)})' \simeq T_i^{(v)}$ for all $v$.
\item[(b)] If $e \neq m$ and $q_{ij}^{(e)} > 0$, then 
$(T_i^{(v)})' \simeq T_i^{(v)}$ for $v \neq e+1$.

\end{itemize}
\end{proposition}

\begin{proof}
(a) is a direct consequence of \ref{l:samecomp}.
For (b) note that by Lemmas \ref{l:samecomp} and \ref{l:modtri} we have  
$(T_i^{(v)})' \simeq T_i^{(v)}$ for $v= 0, \dots, e$ and $v= e+2$.
For $v \geq e+2$ consider the
exchange triangles $$T_i^{(v)} \to B_i^{(v)} \to T_i^{(v+1)} \to.$$
Since $\Hom(T_i^{(v+1)}, T_j^{(1)}[1]) = 0$ by Lemma \ref{l:div2} and $q_{ij}^{(v)} = 0$, it is clear that
the map $T_i^{(v)} \to B_i^{(v)}$ is a left $\add T'/T_i$-approximation. Hence (b) follows.
\end{proof}

\section{Proof of the main result}

This section contains the proof of the main result, Theorem \ref{t:main}. 
As before, let $T = \widetilde{T} \amalg T_i \amalg T_j$ be an 
$m$-tilting object, and let $T' = T/T_j \amalg T_j^{(1)}$.

We will compare the numbers of $(c)$-coloured arrows from $i$ to $k$, 
in the coloured quivers of $T$ and $T'$, i.e. we will compare
$q_{ik}^{(c)}$ and $\tilde{q}_{ik}^{(c)}$.

We need to consider an arbitrary $T$ whose coloured quiver locally looks like  
$$
\xymatrix{
T_i \ar_{(c)}@/_1pc/[rr] \ar^{(e)}[r] & T_j \ar^{(d)}[r] & T_k
}
$$ 
for any possible value of $c, d, e$. Our aim is to show that the formula
$$\tilde{q}_{ik}^{(u)} = 
\begin{cases}  q_{ik}^{(u+1)} & \text{  if $j =k$} \\
		 q_{ik}^{(u-1)} &\text{  if $j=i$} \\
		 \max \{0, q_{ik}^{(u)}  - \sum_{t\neq u} q_{ik}^{(t)} + (q_{ij}^{(u)}-q_{ij}^{u-1}) q_{jk}^{(0)} 
		 + q_{ij}^{(m)} (q_{jk}^{(u)} - q_{jk}^{(u+1)}) \} & \text{  if $i \neq j \neq k$} 
                 \end{cases}
$$
holds. The case where $j=k$ is directly from the definition.  The case
where $j=i$ follows by condition (II) for $Q_{T'}$.    
For the rest of the proof we assume $j \not \in \{i,k \}$.
We will divide the proof into four cases, where $p \geq 0$ denotes the number of arrows from $i$ to $j$,
and $q = q_{jk}^{(0)}$.
\begin{itemize}
\item[I.] $p=0$ 
\item[II.] $p \neq 0$, $e \neq m$ and $q= 0$
\item[III.] $p \neq 0$, $e \neq m$ and $q \neq 0$.   
\item[IV.] $p \neq 0$ and $e=m$
\end{itemize}

Note that in the three first cases, the formula reduces to
$$\tilde{q}_{ik}^{(u)} = \max \{0, q_{ik}^{(u)} - \sum_{t \neq u} q_{ik}^{(t)} + (q_{ij}^{(u)} - q_{ij}^{(u-1)}) q_{jk}^{(0)}  
		  \},$$
and in the first two cases it further reduces to
$$\tilde{q}_{ik}^{(u)} = q_{ik}^{(u)}.$$

\noindent CASE I.
We first consider the situation where there is no coloured arrow $i \to j$, i.e. $q_{ij}^{(u)} = 0$ for all $u$.
That is, we assume $Q_T$ locally looks like this
$$
\xymatrix{
T_i \ar_{(c)}@/_1pc/[rr]  & T_j \ar^{(d)}[r] & T_k
}
$$ 
with $c,d$ arbitrary.
It is a direct consequence of Proposition \ref{p:summarize}
that $q_{ik}^{(u)} = \widetilde{q}_{ik}^{(u)}$ for all $u$ which shows that the formula holds.
\\
\\ \
\noindent CASE II.
We consider the setting
where we assume $Q_T$ locally looks like this
$$
\xymatrix{
T_i \ar_{(c)}@/_1pc/[rr] \ar^{(e)}[r] & T_j \ar^{(d)}[r] & T_k
}
$$ 
with $e \neq m$ and $q=0$.

We then claim that we have the following, which shows that the formula holds.
\begin{lemma}
In the above setting
$q_{ik}^{(u)} = \widetilde{q}_{ik}^{(u)}$ for all $u$.
\end{lemma}

\begin{proof}
It follows directly from Proposition \ref{p:summarize} that $q_{ik}^{(u)} = \widetilde{q}_{ik}^{(u)}$
for $u= 0, \dots, e-1$.
We claim that 
$q_{ik}^{(e)} = \widetilde{q}_{ik}^{(e)}$.

By Lemma \ref{l:comp} we have the (not necessarily minimal) left $\add (T'/T_i)$-approximation

$$T_i^{(e)} \to (T_k)^{pq} \amalg Y^p \amalg X = Y^p \amalg X.$$

First, assume that $T_k$ does not appear as a summand in $B_i^{(e)} = (T_j)^p \amalg X$, then 
the same holds for $Y^p \amalg X$, and hence for $(B_i^{(e)})'$ which
is a direct summand in $Y^p \amalg X$. 

Next, assume $T_k$ appears as a summand in $B_i^{(e)}$, and hence in $X$.
Then $T_k$ is by Proposition \ref{l:disjoint} not a summand in $B_i^{(e+1)}$, and by Lemma \ref{l:modtri}
we have that $T_k$ is also not a summand in $C'$.
Therefore $T_k$ appears with the same multiplicity in $B_i^{(e)}$ as in 
$(B_i^{(e)})'$, also in this case.

We now show that 
$q_{ik}^{(u)} = \widetilde{q}_{ik}^{(u)}$ for $u>e$.

If $q_{ik}^{(e)} \neq 0$, then 
$q_{ik}^{(u)} = \widetilde{q}_{ik}^{(u)} = 0$ for $u>e$ and we are finished.
So assume $q_{ik}^{(e)} = 0$, i.e. $T_k$ does not appear as a direct summand of $X$.

Consider the map $$(T_i^{(e+1)})' \amalg C' \to B_i^{(e+1)} \amalg (T_j^{(1)})^p. $$
We have that $(B_i^{(e+1)})' \simeq \frac{B_i^{(e+1)} \amalg (T_j^{(1)})^p}{C'}$.
By assumption, $T_k$ is  not a direct summand in $(T_k)^{pq} \amalg Y^p \amalg  X = Y^p \amalg X$,
and thus not in $C'$. Hence it follows that $q_{ik}^{(e+1)} = \widetilde{q}_{ik}^{(e+1)}$.

Since, by Proposition \ref{p:summarize} we have for $u= e+2, \dots, m$, that $(T_i^{(u)})' = T_i^{(u)}$ 
and the left $\add (T/ T_i)$-approximation coincide with the left  
$\add (T'/ T_i)$-approximations of $ T_i^{(u)}$, it now follows that
$q_{ik}^{(u)} = \widetilde{q}_{ik}^{(u)}$ for all $u$. 
\end{proof}
\ \\
\noindent CASE III.
We now consider the setting with $p$ non-zero, $q \neq 0$ and $e \neq m$. That is, we assume $Q_T$ locally looks like this
$$
\xymatrix{
T_i \ar_{(c)}@/_1pc/[rr] \ar^{(e)}[r] & T_j \ar^{(0)}[r] & T_k
}
$$ 
where $c \in \{e,e+1 \}$ by Proposition \ref{p:limits}, and where there are $z = q_{ik}^{(c)} \geq 0$ arrows from
$T_i$ to $T_k$.

\begin{lemma}\label{l:formulas}
In the above setting, we have that $Q_{T'}$ is given by

\begin{equation}\label{form1}
\widetilde{q}_{ik}^{(e)} = \begin{cases} 
q_{ij}^{(e)} q_{jk}^{(0)} +  q_{ik}^{(e)}   & \text{ if $c=e$} \\
q_{ij}^{(e)} q_{jk}^{(0)} -  q_{ik}^{(e+1)} & \text{ if $c=e+1$ and $q_{ik}^{(e+1)} \leq q_{ij}^{(e)} q_{jk}^{(0)}$} \\
0                                           & \text{ otherwise}    
\end{cases}
\end{equation}
 
\

\begin{equation}\label{form2}
\widetilde{q}_{ik}^{(e+1)} = \begin{cases} 
- q_{ij}^{(e)} q_{jk}^{(0)} +  q_{ik}^{(e+1)} & \text{ if $c=e+1$ and $q_{ik}^{(e+1)} > q_{ij}^{(e)} q_{jk}^{(0)}$} \\
0                                             & \text{ otherwise}    
\end{cases}
\end{equation}

and

\begin{equation}\label{form3}
\widetilde{q}_{ik}^{(u)} = 0 \text{ for $u \not \in \{e,e+1 \}$} 
\end{equation}
\end{lemma}

\begin{proof}
%

We first deal with the case where $c=e$ and $z>0$. By assumption $X$ in the triangle (\ref{i-tri}) 
has $z$ copies of $T_k$, so  
$(T_k)^{pq} \amalg Y^p \amalg X$ has
$pq +z$ copies of $T_k$. Hence to show (\ref{form1}) it is sufficient to show that $C'$ in the triangle 
$$T_i^{(e+1)} \to (T_i^{(e+1)})' \amalg C' \to (T_j^{(1)})^p \to$$
has no copies of $T_k$. This follows 
directly from the Lemma \ref{l:modtri} and the fact that $T_k$ (by the assumption that $z>0$ and Proposition \ref{l:disjoint}) 
is not a summand in $B_i^{(e+1)}$. In this case (\ref{form2}) and (\ref{form3}) follow directly from Proposition \ref{l:disjoint}.

Consider the case with $c=e+1$ and $0 \leq z \leq pq$. We have that $X$ in the triangle (\ref{i-tri}) 
does not have $T_k$ as a direct summand.
Assume $T_k$ appears as a direct summand of $C'$ with multiplicity $z'$.
We claim that $z' =z$.
Assume first $z' < z$, then on one hand $T_k$ appears with multiplicity $z-z'>0$ in $(B_i^{(e+1)})'$.
On the other hand $T_k$ appears with multiplicity $pq-z' >0$ in $(B_i^{(e)})'$. This contradicts
Proposition \ref{l:disjoint}. Hence $z' =z$.

Therefore $(B_i^{(e)})'$
has $pq-z$ copies of $T_k$ and (\ref{form1}) and (\ref{form2}) hold. 
If $pq \neq z$, then (\ref{form3}) follows directly from the above and Proposition \ref{l:disjoint}. 
In the case $pq = z$, we also need to show that $T_k$ does not appear as a summand in  $(B_i^{(u)})'$
for $u \neq e+1$. Since $pq\ne 0$, we have $z \neq 0$,
and the result follows from Proposition \ref{p:summarize}.

Now assume
$c=e+1$ and $z >pq$.
Assume $C' = (T_k)^l \amalg C''$, where $T_k$ is not a summand in $C''$. Now since
$$(T_i^{(e+1)})' \amalg T_k^l \amalg C'' \to (T_j^{(1)})^p \amalg T_k^z \amalg Y,$$  
with $T_k$ not a summand in $Y$, is a minimal left $\add T'$-approximation, we have
that $l \leq pq < z$ and
$T_k$ appears with multiplicity $z - l > 0$ in the minimal left $\add T'/(T_i)$-approximation of $(T_i^{(e+1)})'$, hence 
$T_k$ cannot appear as a summand in the minimal left $\add T'/T_i$-approximation of $(T_i^{(e)})'$.
Hence $l = pq$, and we have completed the proof of (\ref{form1}) and (\ref{form2}) in this case. 
The case (\ref{form3}), i.e. $u \neq e,e+1$ follows from Proposition \ref{l:disjoint}.
\end{proof}
\ \\
\noindent CASE IV.
We now consider the case with $q_{ij}^{(m)} \neq 0$. 
Assume first there are no arrows from $j$ to $k$. Then we can use 
the symmetry proved in Proposition \ref{p:symmetry} and reduce to case I.
The formula is easily verified in this case.

Assume $d \neq 0$, again we can use the symmetry, this time to reduce to case III.
It is straightforward to verify that the formula holds also in this case.

Assume now that $d= 0$, i.e. we
need to consider the following case
$$
\xymatrix{
T_i \ar@<0.6ex>_{(c)}@/_2.5pc/[rr] \ar@<0.6ex>^{(m)}[r] & T_j \ar@<0.6ex>^{(0)}[r] \ar@<0.6ex>^{(0)}[l] & T_k \ar@<0.6ex>^{(m)}[l] 
\ar@<0.6ex>^{(m-c)}@/^3.5pc/[ll]
}
$$ 
Now by Proposition \ref{p:limits} we have that $c$ is in $\{m, 0 \}$. 
Assume there are $z \geq 0$ $(c)$-coloured arrows

The coloured quiver of $T'$ is of the form
$$
\xymatrix{
T_i \ar@<0.6ex>_{(c')}@/_2.5pc/[rr] \ar@<0.6ex>^{(0)}[r] & T_j^{(1)} \ar@<0.6ex>^{(m)}[r] \ar@<0.6ex>^{(m)}[l] & T_k \ar@<0.6ex>^{(0)}[l] 
\ar@<0.6ex>^{(m-c')}@/^3.5pc/[ll]
}
$$ 
and applying the symmetry of Proposition \ref{p:symmetry} we have that if $z>0$, then
$c' \in \{ 0,m\}$ by Proposition \ref{p:limits}.
Hence for all $u \not \in \{ 0, m \}$ we have that
$\widetilde{q}_{ik}^{(u)} = q_{ik}^{(u)} = 0$.
Therefore it suffices to show that 
$\widetilde{q}_{ik}^{(u)} = q_{ik}^{(u)}$, for $u \in \{ 0,m\}$.
This is a direct consequence of the following. 

\begin{lemma}
Assume we are in the above setting. A map $T_i \to T_k$ or $T_k \to T_i$ is irreducible in $\add T$ 
if and only if it is irreducible in $\add T'$. 
\end{lemma}

\begin{proof}
Assume $T_i \to T_k$ is not irreducible in $\add T'$, and
that $T_i \to T_k = T_i \to U \to T_k$ for some $U = \amalg_t U_t \in \add T'$,
with $U_t$ the indecomposable direct summands of $U$. 
Note that by Lemma \ref{l:div}(a), we can assume that all $T_i \to U_t$ and all $U_t \to T_k$ are non-isomorphisms.
If there is some index $t$ such that $U_t \simeq T_j^{(1)}$,
the map $U_t \to T_k$ factors through some $U'$ in $\add (T/(T_i \amalg T_k))$,
since there are no $(1)$-coloured arrows $j \to i$ or $j \to k$ in the coloured quiver of $T$.
This shows that $T_i \to T_k$ is not irreducible in $\add T$. 

Assume $T_i \to T_k$ is not irreducible in $\add T$,
and that $T_i \to T_k = T_i \to V \to T_k$ for some $V = \amalg_t V_t \in \add T$,
with $V_t$ the indecomposable direct summands of $V$. 
If there is some index $t$ such that $V_t \simeq T_j$,
the map $T_i \to V_t$ factors through $B_j^{(m)}$, which is in $\add (T/(T_i \amalg T_j \amalg T_k)) \subset \add T'$,
since there are no $(0)$-coloured arrows $i \to j$ or $k \to j$ in the coloured quiver of $T$.
This shows that $T_i \to T_k$ is not irreducible in $\add T'$. 

By symmetry, the same property holds for maps $T_k \to T_i$.
\end{proof}

Thus we have proven that the formula holds in all four cases, and this finishes the proof of Theorem \ref{t:main}.

\section{$m$-cluster-tilted algebras}

An $m$-cluster-tilted algebra is an algebra given as $\End_{\C}(T)$
for some tilting object $T$ in an $m$-cluster category $\C =\C_m$.
Obviously, the subquiver of the coloured quiver of $T$ given by the $(0)$-coloured maps is
the Gabriel quiver of $\End_{\C}(T)$.

An application of our main theorem is that the quivers of the $m$-cluster-tilted algebras
can be combinatorially determined via repeated (coloured) mutation.
For this one needs transitivity in the tilting
graph of $m$-tilting objects.
More precisely, we need the following, which is also pointed out in \cite{zz}.

\begin{proposition}
Any $m$-tilting object can be reached from any other 
$m$-tilting object via iterated mutation. 
\end{proposition}

\begin{proof}
We sketch a proof for the convenience of the reader.
Let $T'$ be a tilting object in an $m$-cluster category $\C$ of the hereditary algebra $H=KQ$,
and let $\C_1$ be the $1$-cluster category of $H$. 
By \cite{z}, there is a tilting object $T$ of degree 0, i.e. all direct summands in $T$ have degree 0,
such that $T$ can be reached from $T'$ via mutation. It is sufficient to show that
the canonical tilting object $H$ can be reached from $T$ via mutation.
Since $T$ is of degree 0, it is induced from a $H$-tilting module. Especially $T$  
is a tilting object in $\C_1$. Since $T$ and $H$ are tilting objects in $\C_1$, by \cite{bmrrt} there
are $\C_1$-tilting objects $T= T_0 , T_1, \dots, T_r= H$, such
that $T_i$ mutates to $T_{i+1}$ (in $\C_1$) for $i= 0, \dots, r-1$.
Now each $T_i$ is induced by a tilting module for some $Q_i$ where all $KQ_i$ 
are derived equivalent to $KQ$. Hence, each $T_i$ is easily seen to be an $m$-cluster tilting object.
Since $T_{i+1}$ differs from $T_i$ in only one summand the mutations in $\C_1$ are also mutations in $\C$.
This concludes the proof.
\end{proof}

A direct consequence of the transitivity is the following.

\begin{cor}
For an $m$-cluster category $\C = \C_m$ of the acyclic quiver $Q$, all quivers
of $m$-cluster-tilted algebras are given by repeated coloured mutation of $Q$.
\end{cor}

\section{Combinatorial computation}\label{sec:cc}
In this section, we discuss concrete computation with tilting objects
in an $m$-cluster tilting category. 

An exceptional indecomposable object in $\mod H$ is
uniquely determined by its image $[T]$ in the Grothendieck group $K_0(\mod H)$.  
There is a map from $\mathcal D^b(\mod H)$ to $K_0(\mod H)$ which, for
$T\in \mod H$, takes $T[i]$ to $(-1)^{i}[T]$.  
An exceptional indecomposable in $\mathcal D^b(\mod H)$
can be uniquely specified 
by its class in $K_0(\mod H)$ together with its degree.  

The map from $\mathcal D^b(\mod H)$ to $K_0(\mod H)$ does not descend to 
$\mathcal C$.  However, if we fix our usual choice of fundamental
domain in $\mathcal D^b(\mod H)$, then we can identify the indecomposable
objects in it as above.  

Let us define the combinatorial data corresponding to a tilting object
$T$ to be $Q_T$ together with 
$([T_i], \deg T_i)$ for $1\leq i \leq n$.

\begin{thm} Given the combinatorial data for a tilting object $T$
in $\mathcal C$,
it is possible to determine, by a purely combinatorial procedure, 
 the combinatorial data for the tilting object
which results from an arbitrary sequence of mutations applied to $T$.  
\end{thm}

\begin{proof} Clearly, it suffices to show that, for any $i$, 
we can determine the class and degree for $T_i^{(j)}$.  If we can do
that then,
by
the coloured mutation procedure, we can determine the coloured
quiver for $(T/T_i)\amalg T_i^{(j)}$, and by applying this procedure
repeatedly, we can calculate the result of an arbitrary sequence of mutations.

Since we are given $Q_T$, we know $B_i^{(0)}$, and we can calculate 
$[B_i^{0}]$.  Now we have the following lemma:

\begin{lemma}\label{one} $[T^{(1)}_i]=[B_i^{(0)}]-[T^{(0)}_i]$, and 
$\deg(T^{(1)}_i)=\deg T_i$ or $\deg T_i+1$, whichever is consistent with the sign of the
class of $[T^{(1)}_i]$, unless this yields a non-projective indecomposable
object in degree $m$, or an indecomposable of degree $m+1$. \end{lemma}

\begin{proof} The proof is immediate from the exchange triangle  
$T_i \rightarrow B_i^{(0)}\rightarrow T_i^{(1)}\rightarrow$.
\end{proof}

Applying this lemma, and supposing that we are not in the case where its
procedure fails, we can determine the class and degree $T^{(1)}_i$.  By the coloured
mutation procedure, we can also determine the coloured quiver for 
$\mu_i(T)$.  We therefore have all the necessary data to apply Lemma~\ref{one}
again.  Repeatedly applying the lemma, there is some $k$ such that we
can calculate the class and degree of $T_i^{(j)}$ for $1\leq j\leq k$,
and the procedure described in the lemma fails to calculate
$T_i^{(k+1)}$.  

We also have the following lemma:

\begin{lemma}\label{two} $[T^{(m)}_i]=[B_i^{(m)}]-[T^{(0)}_i]$, 
and $\deg T^{(m)}_i=
\deg T_i$ or $\deg T_i-1$, whichever is consistent with the sign of 
$[T^{(m)}_i]$, unless this yields an indecomposable in degree $-1$.  
\end{lemma}

Applying this lemma, starting again with $T$, we can obtain the degree
and class for $T_i^{(m)}$.  We can then determine the coloured
quiver for $\mu_i^{-1}(T)$, and we are now in a position to apply 
Lemma~\ref{two} again.  
The last complement which
Lemma~\ref{two} will successfully determine is $T^{(k+1)}_i$.  
It follows that we can determine the degree and class of any complement to
$T/T_i$.  
\end{proof}

\section{The $m$-cluster complex}

In this section, we discuss the application of our results to the 
study of the $m$-cluster complex, a simplicial complex defined in \cite{FR} 
for a finite root system $\Phi$.  
We shall begin by stating our results for the $m$-cluster complex
in purely combinatorial language,
and then briefly describe how they follow from the representation-theoretic
perspective in the rest of the paper.  For simplicity, we restrict to the
case where $\Phi$ is simply laced.


Number the vertices
of the Dynkin diagram for $\Phi$ from 1 to $n$.  
The $m$-coloured almost positive roots, $\mP$, consist
of $m$ copies of the positive roots, numbered $1$ to $m$, together
with a single copy of the negative simple roots.  We refer to an element
of the $i$-th copy of $\Phi^+$ as having colour $i$, and we write such an
element as $\beta^{(i)}$.  

Since the Dynkin diagram for $\Phi$ is a tree, it is bipartite; 
we fix 
a bipartition $\{1,\dots,n\}=I_+\cup I_-$.

The $m$-cluster complex, $\mD$, is a simplicial complex on the ground set 
$\mP$.  Its maximal faces are 
called $m$-clusters.  The definition of $\mD$ 
is combinatorial; we refer the reader
to \cite{FR}.  The $m$-clusters each consist of $n$ elements of
$\mP$ \cite[Theorem 2.9]{FR}.
Every codimension 1 face of $\mD$ is contained in exactly
$m+1$ maximal faces \cite[Proposition 2.10]{FR}.   
There is a certain combinatorially-defined bijection
$\mR \colon \mP\rightarrow \mP$, which takes faces of $\mD$ to
faces of $\mD$ \cite[Theorem 2.4]{FR}.

It will be convenient to consider {\it ordered $m$-clusters}.  An ordered
$m$-cluster is just a $n$-tuple from $\mP$, 
the set of whose elements forms an $m$-cluster.
Write $\mS$ for the set of ordered $m$-clusters.

For each ordered $m$-cluster $C=(C_1,\dots,C_n)$, we will define a coloured quiver $Q_C$.  
We will
also define an operation $\mu_j:\mS\rightarrow \mS$, which takes 
ordered $m$-clusters to ordered $m$-clusters, changing only the $j$-th
element.

We will define both operations inductively.  The set $-\Pi$ of negative simple 
roots forms an $m$-cluster.  Its associated quiver is defined by drawing, for
each edge $ij$ in the Dynkin diagram, a pair of arrows.  Suppose $i\in I_+$
and $j \in I_-$. Then we draw an arrow from $i$ to $j$ with colour 
$0$, and an arrow from $j$ to $i$ with colour $m$.  

Suppose now that we have some ordered $m$-cluster $C$, together with its
quiver $Q_C$.  We will now proceed to define $\mu_j(C)$.
Write $q_{jk}^{(0)}$ for the number of arrows in $Q_C$ 
of colour
$0$ from $j$ to $k$.  
Define: 

$$\beta= -C_j+\sum_{k\ne j} q_{jk}^{(0)}C_k.$$

Let $c$ be the colour of $C_j$.  We define $\mu_j(C)$ by replacing
$C_j$ by some other element of $\mP$, according to the following rules:

\begin{itemize}
\item
If $C_j$ is positive and $\beta$ is positive, replace $C_j$ by $\beta^{(c)}$.
\item
If $C_j$ is positive and $\beta$ is negative, replace 
$C_j$ by $\mR(-\beta^{(c)})$.  
\item
If $C_j$ is negative simple $-\alpha_i$, define $\gamma$ by 
$\gamma^{(0)}=\mR(-\alpha_i)$, and then replace $C_j$ by 
$\beta+C_j-\gamma$, with colour zero.  
\end{itemize}
Define the quiver for the $m$-cluster $\mu_j(C)$ by the coloured quiver
mutation rule from Section 2.  
Since any $m$-cluster can be obtained from $-\Pi$ by a sequence of 
mutations, the above suffices to define $\mu_j(C)$ and $Q_C$ for any
ordered $m$-cluster $C$.  

\begin{proposition} The operation $\mu_j$ defined above takes 
$m$-clusters to $m$-clusters, and the $m$-clusters 
$\mu_j^{i}(C)$ for $0\leq i \leq m$ are exactly those containing all the 
$C_i$ for $i\ne j$.  
\end{proposition}

The connection between the combinatorics discussed here and the representation
theory in the rest of the paper is as follows.  $\mP$ corresponds to 
the indecomposable objects
of (a fundamental domain for) $\C_m$.  
The cluster tilting objects in $\C_m$  
correspond to the $m$-clusters.  The operation $\mR$ corresponds to $[1]$.  
For further details on the translation, the reader is referred to 
\cite{t,z}.  The above proposition then follows from the approach
taken in Section \ref{sec:cc}.

\section{An alternative algorithm for coloured mutation}

Here we give an alternative description of coloured quiver mutation at vertex $j$.  

\begin{enumerate}
\item For each pair of arrows
$$ \xymatrix {i \ar^{(c)}[r] & j\ar^{(0)}[r] &k }$$
with $i\ne k$, the arrow from $i$ to $j$ of arbitrary colour $c$, and the arrow from
$j$ to $k$ of colour $0$, add a pair of arrows: an arrow from $i$ to $k$
of colour $c$, and one from $k$ to $i$ of colour $m-c$.  
\item If the graph violates property II, because for some pair of vertices $i$ and $k$ there are arrows from
$i$ to $k$ which have two different colours, cancel the same number of 
arrows of each colour, until property II is satisfied.
\item Add one to the colour of any arrow going into $j$ and subtract one 
from the colour of any arrow going out of $j$.
\end{enumerate}


\begin{proposition} The above algorithm is well-defined and 
correctly calculates coloured 
quiver mutation as previously defined.
\end{proposition}

\begin{proof} 
Fix a quiver $Q$ and a vertex $j$ at which the mutation is being carried out.

To prove that the algorithm is well-defined, we must show that at step 2,
there are only two colours of arrows running from $i$ to $k$ for any
pair of vertices $i$, $k$. (Otherwise there would be more than one way to
carry out
the cancellation procedure of step 2.)    

Since in the original quiver $Q$, there was only one colour of arrows from
$i$ to $k$, in order for this problem to arise, we must have added two
different colours of arrows from $i$ to $k$ at step 1.  
Two colours of arrows will only be added from
$i$ to $k$ if, in $Q$,  there are both 
$(0)$-coloured arrows from $j$ to $k$ and from
$j$ to $i$.  In this case, by property III, there are $(m)$-coloured arrows
from $i$ to $j$ and from $k$ to $j$.  It follows that in step 1, we will
add both $(0)$-coloured and $(m)$-coloured arrows.  
Applying Proposition 5.1, we see that any arrows from $i$ to $k$ in 
$Q$ are of colour 0 or $m$.   Thus, as desired, after step 1, there are 
only two colours of arrows in the quiver, so step 2 is well-defined.  

We now prove correctness.  
Let $\widetilde Q=\mu_j(Q)$.  
Write $q_{ij}^{(c)}$ for the number of $c$-coloured arrows from 
$i$ to $j$ in $Q$, and similarly $\widetilde q_{ij}^{(c)}$ for $\widetilde
Q$.  Write $\hat Q$ and $\hat q_{ij}^{(c)}$ for the result of applying the
above algorithm.  

It is clear that only the final step of the algorithm is relevant for 
$\hat q_{ik}$ where one of $i$ or $k$ coincides with $j$, and 
therefore that in this case $\hat q_{ij}^{(c)}=\widetilde q_{ij}^{(c)}$
as desired.  

Suppose now that neither $i$ nor $k$ coincides with $j$.  Suppose further that
in $Q$ there
are no $(0)$-coloured arrows from either $i$ or $k$ to $j$, and therefore also no
$m$-coloured arrows from $k$ to $i$ or $j$.  In this case, 
$\widetilde q_{ik}^{(c)} = q_{ik}^{(c)}$.
In the algorithm, no arrows will
be added between $i$ and $k$ in step 1, and therefore no further changes
will be made in step 2.  Thus $\hat q_{ik}^{(c)} = q_{ik}^{(c)}
= \widetilde q_{ik}^{(c)}$, as desired.

Suppose now that there are $(0)$-coloured arrows from $j$ to both $i$ and 
$k$.  In this case, $\widetilde q_{ik}^{(c)}=q_{ik}^{(c)}$.  
In this case, as discussed in the proof of well-definedness, an
equal number of $(0)$-coloured and $(m)$-coloured arrows will be introduced
at step 1.  They will therefore be cancelled at step 2.  
Thus $\hat q_{ik}^{(c)}=q_{ik}^{(c)}=
\widetilde q_{ik}^{(c)}$ as desired.

Suppose now that there is a $(0)$-coloured arrow from $j$ to $k$, but 
not from $j$ to $i$. Let the arrows from
$i$ to $j$, if any, be of colour $c$.  
At step 1 of the algorithm, we will add $q_{ij}^{(c)}q_{jk}^{(0)}$ arrows of colour 
$c$ to $Q$.  By Proposition 5.1, the arrows in $Q$ from $i$ to $k$ are of
colour $c$ or $c+1$.  
One verifies that the algorithm yields the same result as coloured
quiver mutation, in the three cases that the arrows from $i$ to $k$ in
$Q$ are of colour $c$, that they are of colour $c+1$ but there are fewer
than $q_{ij}^{(c)}q_{jk}^{(0)}$, 
and that they are of colour $c+1$ and there are at least
as many as $q_{ij}^{(c)}q_{jk}^{(0)}$.  

The final case, that there is a $(0)$-coloured arrow from $j$ to $i$
but not from $j$ to $k$, is similar to the previous one.  
\end{proof}

\section{Example: type $A_n$}

In \cite{bm1}, a certain category $\CBM$ is constructed, which is 
shown to be equivalent to the $m$-cluster category of Dynkin type $A_n$.  
The description of $\CBM$ is as follows.  
Take an $nm+2$-gon $\Pi$, with vertices labelled clockwise from 
1 to $nm+2$.  Consider
the set $X$ of diagonals $\gamma$ of $\Pi$ with the property that 
$\gamma$ divides $\Pi$ into two polygons each having a 
number of sides congruent
to 2 modulo $m$.  For each $\gamma\in X$, there is an object $A_\gamma$ in
$\CBM$.  These objects $A_\gamma$ form the indecomposables of the
additive category $\CBM$.  We shall not recall the exact definition of
the morphisms, other than to note that they are generated by the morphisms
$p_{ijk}:A_{ij} \rightarrow A_{ik}$ 
which exist provided that $ij$ and $ik$ are both diagonals in $X$,
and that, starting at $j$ and moving clockwise around $\Pi$, 
one reaches $k$ before $i$.

A collection of diagonals in $X$
is called non-crossing if its elements intersect pairwise only on the 
boundary of the polygon.  
An inclusion-maximal such collection of diagonals divides 
$\Pi$ into $m+2$-gons; we therefore refer to such a collection of
diagonals as an $m+2$-angulation.  
If we remove one diagonal $\gamma$ from an $m+2$-angulation $\Delta$, 
then the two $m+2$-gons on either side of
$\gamma$ become a single $2m+2$-gon.  We say that $\gamma$ is a {\it diameter}
of this $2m+2$-gon, since it connects vertices which are diametrically
opposite (with respect to the $2m+2$-gon).  If $\delta$ is 
another diameter of this $2m+2$-gon, then 
$(\Delta \setminus \gamma)\cup \delta$
is another maximal noncrossing collection of diagonals from $X$.  
(In particular, $\delta\in X$.)  

For $\Delta$ an $m+2$-angulation, let 
$A_\Delta=\bigamalg_{\gamma\in\Delta} A_\gamma$.  Then we have that 
$A_\Delta$ is a
basic ($m$-cluster-)tilting object for $\CBM$, 
and all basic tilting objects of 
$\CBM$ arise in this way.  
It follows from the previous discussion that if $T=A_\Delta$ is a 
basic tilting object, and $\gamma\in \Delta$, then the complements to
$A_{\Delta\setminus \gamma}$ will consist of the objects $A_\delta$ 
where $\delta$ is a diameter of the $2m+2$-gon obtained by removing 
$\gamma$ from the $m+2$-angulation determined by $\Delta$.  In fact,
we can be more precise.  Define $\delta_{(i)}$ to be the diameter 
of the $2m+2$-gon obtained by rotating the vertices of $\gamma$ by $i$ steps 
counterclockwise (within the $2m+2$-gon).  Then $A_\gamma^{(i)}=A_{\delta_{(i)}}$.

Define $Q_\Delta$ to be the coloured quiver for the tilting object
$A_\Delta$,
with vertex set $\Delta$.
Using the setup of \cite{bm1}, it is straightforward to verify:

\begin{proposition} The coloured quiver $Q_{\Delta}$ of $T=A_\Delta$ 
has an arrow from $\gamma$ to $\delta$ if and only if $\gamma$ and $\delta$ both lie
on some $m+2$-gon in the $m+2$-angulation defined by $\Delta$.  In this
case, the colour of the arrow is the number of edges forming the segment of the boundary of the 
$m+2$-gon which lies between $\gamma$ and $\delta$, counterclockwise
from $\gamma$ and clockwise from $\delta$.  
\end{proposition}

Given the proposition above, 
it is straightforward to verify directly that
$Q_\Delta$ satisfies conditions (I), (II), and (III), and that 
mutation is indeed given by the mutation of coloured quivers.  

\bigskip
{\bf Example: $A_3$, $m= 2$}

We return to the example from Section 2.  The quadrangulation of a 
decagon corresponding to the tilting object $T$ is on the
left.  The quadrangulation corresponding to $T'$ is on the
right.  Passing from the figure on the left to the figure
on the right, the diagonal 27 (which corresponds to the 
summand $I_2$) has been rotated one step counterclockwise within the hexagon
with vertices 1,2,3,4,7,10.

$$\epsfbox{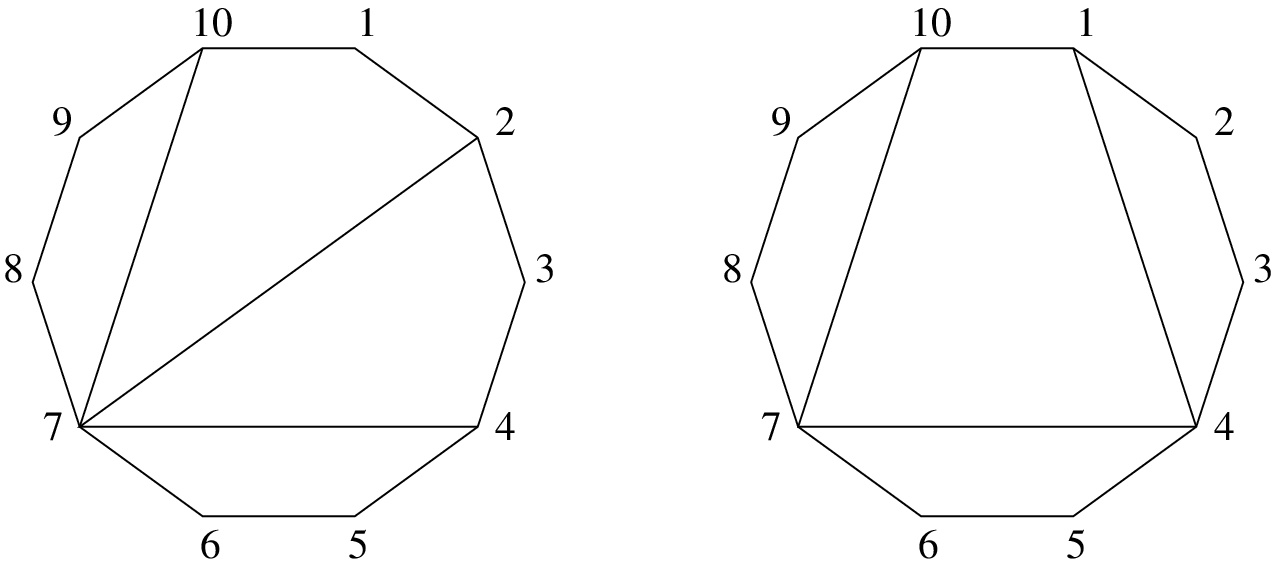}$$

%

\end{document}